\documentclass[a4paper, 10pt]{amsart}
\usepackage{dsfont}
\usepackage[all]{xy}
\usepackage{hyperref}
\usepackage{graphicx}

\title{Whitney categories and the Tangle Hypothesis}

\author{Conor Smyth and Jon Woolf}

\date{August 2011}

\thanks{
This work was made possible by the generous support of the Leverhulme Trust (Grant ref. F/00 025/AI). The second author would also like to thank the Newton Institute, Cambridge for their hospitality and support in April and May 2011, whilst this paper was in preparation, and Tom Leinster for several helpful conversations.}

\newtheorem{theorem}{Theorem}[section]
\newtheorem{proposition}[theorem]{Proposition}
\newtheorem{corollary}[theorem]{Corollary}
\newtheorem{lemma}[theorem]{Lemma}

\theoremstyle{definition}
\newtheorem{definition}[theorem]{Definition}
\newtheorem{convention}[theorem]{Convention}
\newtheorem{example}[theorem]{Example}
\newtheorem{examples}[theorem]{Examples}
\newtheorem{remark}[theorem]{Remark}

\newcommand{\defn}[1]{\emph{#1}}

\newcommand{\ie}{i.e.\ }

\newcommand{\N}{\mathbb{N}}

\newcommand{\R}{\mathbb{R}}

\newcommand{\id}{\mathrm{id}}

\newcommand{\unit}{\mathds{1}}

\newcommand{\pt}{{\rm pt}}

\DeclareMathOperator{\colim}{colim}

\newcommand{\cat}[1]{\mathsf{#1}} 
\newcommand{\sets}{\cat{Sets}}
\newcommand{\dagcat}{\cat{Dagger}}
\newcommand{\whit}[1]{{#1}\cat{Whit}}
\newcommand{\whitify}{\omega}
\newcommand{\wcat}[1]{\cat{#1}} 

\newcommand{\prestrat}[1]{\cat{Prestrat}_{#1}} 
\newcommand{\strat}[1]{\cat{Strat}_{#1}} 

\newcommand{\thcat}[2]{\Psi_{#1}\!\left(#2\right)} 
\newcommand{\tang}[2]{{#1}\wcat{Tang}^\textrm{fr}_{#2}} 
\newcommand{\gtang}[2]{{#1}\wcat{Tang}^G_{#2}} 

\newcommand{\sphere}[1]{\mathbb{S}^{#1}} 
\newcommand{\thom}[1]{\mathbb{M}{#1}} 
\newcommand{\cube}[1]{\mathbb{I}^{#1}} 

\newcommand{\rep}[1]{\wcat{Rep}\!\left(#1\right)} 

\newcommand{\presh}[1]{\cat{PreSh}\left(#1\right)} 
\newcommand{\sh}[1]{\mathcal{#1}} 

\newcommand{\sieve}[1]{\mathcal{#1}}
\newcommand{\nat}[2]{\mathrm{Nat}\left(#1,#2\right)}

\begin{document}

\maketitle
 
 \section{Introduction}
  \label{introduction}
  
 We propose a new notion of `$n$-category with duals', which we call a Whitney $n$-category.  There are two motivations. The first is to give a definition which makes the Baez--Dolan Tangle Hypothesis \cite{MR1355899} almost tautological. The Tangle Hypothesis is that, given suitable definitions of the terms in quotes,
 \begin{quote}
The `$n$-category of framed codimension $k$ tangles' is equivalent to the `free $k$-tuply monoidal $n$-category with duals on one object'.
\end{quote}
 This generalises Shum's theorem \cite{MR1268782} that the category of framed tangles in three dimensions is equivalent to the free tortile tensor category on one object.   In \S\ref{tangle hypothesis} we prove a version of the hypothesis by interpreting it as a statement about Whitney $n$-categories.   
 There is of course a price to pay for obtaining a simple proof of the Tangle Hypothesis, and that is that  Whitney $n$--categories are a geometric, as opposed to algebraic, theory of higher categories. Therefore to realise more fully the original conception one should relate Whitney categories to some more algebraic theory of higher categories. Sadly this is not something we understand how to do at this stage. 
 
 The second motivation, in fact the original one for this work, is to give a definition which enables us to construct `fundamental $n$-categories with duals' for each smooth stratified space. The idea here, also due to Baez and Dolan, is that there should be a variant of homotopy theory which detects aspects of the stratification of a stratified space. The invariants will not be groupoids but rather more general categories with duals (a groupoid is a category with duals with the additional property that the dual of a morphism is an inverse). They are obtained by restricting attention to maps into the space which are transversal to all strata; the full construction, and the functoriality, of the invariants is explained in \S\ref{transversal homotopy theory}.

The definition of Whitney category has a geometric flavour, and is intended for applications in smooth geometry. We borrow heavily from the ideas of Morrison and Walker expressed in \cite{Morrison:2011fk}. They promote the point of view that
\begin{enumerate}
\item it should be easier to define a notion of $n$-category with duals than of plain $n$-category;
\item one should consider higher morphisms of quite general shapes (not merely globules, simplices, or cubes);
\item rather than having a source and target, a morphism should have a `boundary' encompassing both.
\end{enumerate}
To emphasise the first point; Whitney categories are not a general theory of higher categories, but only a theory of `higher categories with duals'. This fragment of higher category theory appears to be simpler and more amenable to a geometric treatment. Despite Morrison and Walker's influence, our definition of Whitney category is quite different from their definition of topological or disk-like category. They give an inductive list of axioms, whereas we define an $n$-category as a presheaf of sets on a category $\prestrat{n}$ of stratified spaces and {\em prestratified} maps, whose restriction to the subcategory $\strat{n}$ of stratified spaces and {\em stratified} maps is a sheaf for a certain Grothendieck topology. The subscript $n$ refers to the fact that we consider only spaces of dimension $\leq n$, and that the morphisms are homotopy classes of maps relative to the strata of dimension $<n$. Roughly, by stratified space we mean a Whitney stratified space with cellular strata, by a stratified map we mean a smooth map whose restriction to each stratum in the source is a locally-trivial fibre bundle over a stratum in the target, and by a prestratified map we mean one which becomes stratified after a possible subdivision of the stratification of the source. The precise definitions, as well as the specification of the topology on $\strat{n}$, are the subject of \S\ref{stratified spaces}. 

The definition of Whitney category appears in \S\ref{whitney cat}. We consider Whitney categories as a full subcategory $\whit{n}$ of the presheaves on $\prestrat{n}$. Various formal properties follow; $\whit{n}$ is complete, cocomplete and there is a left adjoint to the inclusion into the presheaves, which associates a Whitney category to any presheaf. We also introduce a notion of equivalence of Whitney $n$-categories --- Definition \ref{equivalent} --- generalising the description of an equivalence of (ordinary) categories as a span of fully-faithful functors which are surjective on objects.

At first sight the notion of Whitney category is quite remote from the usual notion of category.  The intuitive picture is as follows. The set $\wcat{A}(X)$ associated to the space $X$ consists of the `morphisms in $\wcat{A}$ of shape $X$'. For example the point-shaped morphisms $\wcat{A}(\pt)$ are the objects. All our spaces carry specified stratifications, and these play an important r\^ole. For example the set associated to an interval stratified by its endpoints and interior is the $1$-morphisms, whereas the set associated to the subdivided interval with a third point stratum in the interior is the set of pairs of composable $1$-morphisms. This last assertion uses the fact that a Whitney category is a sheaf on $\strat{n}$. More generally, insisting that a Whitney category is a sheaf  ensures that the set it assigns to a space $X$ is determined by the sets it assigns to the (cellular) strata. One can think of $X$ as a template for pasting diagrams, and the set assigned to $X$ as the set of pasting diagrams in $\wcat{A}$ modelled on this template. Prestratified maps between spaces induce maps, in the opposite direction, between the corresponding sets. In particular, 
\begin{itemize}
\item the inclusion of the boundary induces a map taking a cell-shaped morphism to its `boundary', which plays the r\^ole of source and target combined;
\item the map to a point induces a map taking an object in $\wcat{A}(\pt)$ to the identity morphism (of appropriate shape) on that object;
\item a subdivision of a cell induces a map taking a pasting diagram modelled on the subdivided cell to its composite.
\end{itemize}
To further clarify the relation consider the simplest case of Whitney $0$-categories. Since $\prestrat{0}$ contains only $0$-dimensional spaces, the only information here is the set of objects $\wcat{A}(\pt)$. More precisely the map $\wcat{A} \mapsto \wcat{A}(\pt)$ induces an equivalence between the category of Whitney $0$-categories and the category of sets. The next simplest case, $n=1$, is treated in \S \ref{relation to ordinary categories}, where we show that the category of Whitney $1$-categories and the category of small dagger categories are equivalent. The sets $\wcat{A}(\pt)$ and $\wcat{A}([0,1])$ are respectively the objects and morphisms of the dagger category corresponding to $\wcat{A}$. This correspondence is our principal justification for considering Whitney $n$-categories as `$n$-categories with duals'.
 
Many of our examples will be {\em $k$-tuply monoidal} Whitney $n$-categories. By such we mean a Whitney $(n+k)$-category which is `trivial' in dimensions $<k$, \ie that assigns a one element set to any space $X$ with $\dim X < k$. This slightly confusing terminology makes sense if one recalls that a monoid can be viewed as a one-object category, a commutative monoid a one object, one morphism bicategory and so on. In \S\ref{monoidal whitney categories} we give a functorial procedure for associating a genuine Whitney $n$-category $\Omega^k\wcat{A}$ to a $k$-tuply monoidal one $\wcat{A}$, by considering the presheaf $\wcat{A}(\sphere{k} \times -)$ where $\sphere{k}$ is the $k$-sphere stratified by a point and its complement. This is the analogue of the re-indexing procedure used to turn a one-object category into a monoid.

Three classes of examples are discussed in \S\ref{examples}. Firstly, we show that representable presheaves are Whitney categories.  Secondly, we define a $k$-tuply monoidal Whitney $n$-category $\tang{n}{k}$ of framed tangles. The $X$-shaped morphisms are the set of framed codimension $k$ submanifolds of $X$, transversal to all strata, considered up to isotopies relative to strata of dimension $< n+k$. Interpreted in this framework the Tangle Hypothesis says: 
\begin{quote}
The Whitney category $\tang{n}{k}$ of framed tangles is equivalent to the {\em free} $k$-tuply monoidal Whitney $n$-category on one $\sphere{k}$-morphism. 
\end{quote}
We prove this in \S \ref{tangle hypothesis} by establishing an equivalence between $\tang{n}{k}$ and the Whitney $(n+k)$-category represented by the sphere $\sphere{k}$. This equivalence arises from the Pontrjagin--Thom collapse map construction which relates framed codimension $k$ tangles in $X$ to maps $X \to \sphere{k}$. The Whitney category represented by the sphere is, by the Yoneda Lemma, free on one $\sphere{k}$-morphism, namely the identity map of the sphere.

The third class of examples is provided by transversal homotopy theory: in \S\ref{transversal homotopy theory} we explain how to associate a transversal homotopy Whitney category $\thcat{k,n+k}{M}$ to each based Whitney stratified manifold $M$.  The $X$-shaped morphisms are the set of transversal maps $X\to M$ considered up to homotopy relative to strata in $X$ of dimension $< n+k$. We also insist that the maps are `based' in that strata in $X$ of dimension $<k$ are mapped to the basepoint. This makes $\thcat{k,n+k}{M}$ into a $k$-tuply monoidal Whitney $n$-category. For $n=0$ and $n=1$ these are closely related respectively to the transversal homotopy  monoids and the transversal homotopy categories introduced in \cite{MR2720181}. See \S \ref{transversal homotopy theory} and Example \ref{transversal homotopy relationship}  for details of the respective relationships. The use of Whitney categories thus allows us to extend the definitions of \cite{MR2720181} to arbitrary $n$, and provides a general framework for studying transversal homotopy theory. 

The transversal homotopy theory of spheres is also closely related to framed tangles. In \S \ref{transversal homotopy of spheres} we show that it is equivalent to the Whitney category represented by the sphere. Thus we have equivalences of $k$-tuply monoidal Whitney $n$-categories
$$
\thcat{k,n+k}{\sphere{k}} \simeq \rep{\sphere{k}} \simeq \tang{n}{k}
$$
yielding three descriptions of the same object which we can think of respectively as homotopy-theoretic, algebraic (in the sense that the representable Whitney category is free on one $\sphere{k}$-morphism) and geometric.

The final section \S\ref{other flavours} contains some remarks about the Tangle Hypothesis for tangles with other normal structures, and the relationship of these with transversal homotopy theory of Thom spaces other than the sphere.

Our examples and applications are in smooth geometry (smooth tangles, transversal homotopy theory, \ldots) so we have developed a smooth theory of $n$-categories with duals based on Whitney stratified spaces. This choice is not essential. Firstly, it is not clear that we need the Whitney conditions; the theory could be developed using the weaker notion of smooth spaces with manifold decompositions. However, the Whitney conditions are required to obtain a good theory of stratified smooth spaces, for instance to ensure that transversal maps form an open dense subset of all smooth maps. Since transversality plays a central r\^ole it seems  natural to impose the Whitney conditions, particularly when considering transversal homotopy theory. More generally, there seems no reason why one should not develop an analogous theory by starting instead with stratified PL spaces, or subanalytic ones or indeed any of a number of other choices. It would also be interesting to replace Whitney stratified spaces by a `combinatorial' category, for instance by symmetric simplicial sets. A better understanding of combinatorial versions of this theory seems the most likely way of building a bridge to Lurie's theory of $(\infty,n)$-categories with adjoints, and his proofs of the Tangle and Cobordism hypotheses \cite{Lurie:2009fk}.

\section{Stratified spaces and maps}
\label{stratified spaces}

\subsection{Whitney stratified spaces}

A \defn{stratification} of a smooth manifold $M$ is a  decomposition $M = \bigcup_{i\in \mathcal{S}} S_i$ into disjoint subsets $S_i$ indexed by a poset $\mathcal{S}$ such that
\begin{enumerate}
\item the decomposition is locally-finite,
\item $S_i \cap \overline{S_j} \neq \emptyset \iff S_i \subset \overline{S_j}$, and this occurs precisely when $i \leq j$ in $\mathcal{S}$,
\item each $S_i$ is a locally-closed smooth connected submanifold of $M$.
\end{enumerate}
The $S_i$ are referred to as the \defn{strata} and the partially-ordered set $\mathcal{S}$ as the \defn{poset of strata}. The second condition is usually called the \defn{frontier condition}. 

Nothing has been said about how the strata fit together from the point of view of smooth geometry. To govern this we impose further conditions, proposed by  Whitney \cite{MR0188486} following earlier ideas of Thom \cite{MR0239613}. Suppose $x\in S_i \subset \overline{S_j}$ and that we have sequences $(x_k)$ in $S_i$ and $(y_k)$ in $S_j$ converging to $x$. Furthermore, suppose that the secant lines $\overline{x_ky_k}$ converge to a line $L\leq  T_xX$ and the tangent planes $T_{y_k}S_j$ converge to a plane $P \leq T_xM$. (An intrinsic definition of the limit of secant lines can be obtained by taking the limit of $(x_k,y_k)$ in the blow-up of $M^2$ along the diagonal, see \cite[\S4]{Mather:1970fk}. The limit of tangent planes is defined in the Grassmannian $Gr_d(TM)$ where $d=\dim S_j$. The limiting plane $P$ is referred to as a \defn{generalised tangent space} at $x$.) In this situation we require
 \begin{description}
\item[(Whitney A)] the tangent plane $T_xS_i$ is a subspace of the limiting plane $P$;
\item[(Whitney B)] the limiting secant $L$ is a subspace of the limiting plane $P$.
\end{description}
Mather \cite[Proposition 2.4]{Mather:1970fk} showed that the second Whitney condition implies the first. Nevertheless, it is useful to state both conditions because the first is often what one uses in applications, whereas the second is necessary  to ensure that the normal structure to a stratum is locally topologically trivial, see for example \cite[1.4]{GoM2}.

A \defn{Whitney stratified manifold} is a manifold with a stratification satisfying the Whitney B condition. A 
\defn{Whitney stratified space}  is a closed union of strata $X$ in a Whitney stratified manifold $M$.  Examples abound, for instance any manifold with the \defn{trivial stratification} which has only one stratum is a Whitney stratified manifold. More interestingly, any complex analytic variety admits a Whitney stratification \cite{MR0188486}, indeed any (real or complex) subanalytic set of an analytic manifold admits a Whitney stratification \cite{MR0377101,H}.

A continuous map $f:X\to Y$ of Whitney stratified spaces is \defn{smooth} if it extends to a smooth map of the ambient manifolds. The notion of smoothness depends only on the germ of the ambient space, in fact only on the equivalence class of the germ generated by embeddings into larger ambient spaces. By embedding the manifold $M$ we may always assume that the ambient space of $X$ is Euclidean. 
\begin{definition}
A \defn{stratified smooth space} $X$ is the stable germ of a compact Whitney stratified subspace of some $\R^k$, where we stabilise by the standard inclusions $\R^k \hookrightarrow \R^{k+1} \hookrightarrow \cdots$. We abuse notation by using the same letter to denote the germ and the underlying Whitney stratified space. A \defn{smooth map of stratified smooth spaces} is the stable germ of a smooth map, where we stabilise by taking products with $\R$.
\end{definition}

We will restrict our attention to stratified spaces glued together from cells: a \defn{cellular stratified space} is a stratified space $X$ in which each stratum $S$ is contractible.
\begin{examples}
\label{strat ex}
\begin{enumerate}
\item Let $\cube{}$ be the germ of $\R$ along the interval $[0,1]$ stratified by the endpoints and interior. Similarly let $\cube{n}$ be the germ of $\R^n$ along $[0,1]^n$ stratified in the obvious fashion by faces. 
\item Let $\sphere{n}$ be the sphere $S^n$ stratified by a point, call it $0$, and its complement and considered as a germ of $\R^{n+1}$. 
\end{enumerate}
\end{examples}

\subsection{Stratified maps}

We are not interested in all smooth maps, but only those which interact nicely with the given stratifications. 
\begin{definition}
\label{strat sub}
A smooth map $f:X \to Y$ is a \defn{stratified submersion} if for any stratum $B$ of $Y$ 
\begin{enumerate}
\item the inverse image $f^{-1}B$ is a union of strata of $X$ and
\item for any stratum $A \subset f^{-1}B$ the restriction $f|_A : A \to B$ is a submersion.
\end{enumerate}
\end{definition}
Whether or not a smooth map is stratified depends only upon the map of underlying spaces, and not on the germ. The composite of stratified submersions is a stratified submersion. Thom's first isotopy lemma implies that the restriction $f|_{f^{-1}B} : f^{-1}B \to B$ is topologically a locally trivial fibre bundle. 
 
\begin{convention}
For ease of reading, in the sequel we refer to stratified smooth spaces and stratified submersions simply as \defn{stratified spaces} and \defn{stratified maps}. A map is \defn{weakly stratified} if it obeys only the first condition of Definition \ref{strat sub}.
\end{convention}
Stratified maps $f,g:X \to Y$ are \label{htpy} \defn{homotopic through stratified maps relative to strata of dimension $<n$} if there is a smooth map germ $h: X \times [0,1] \to Y$ such that 
\begin{enumerate}
\item each $h(-,t) : X \to Y$ is stratified and
\item $h(x,t) = h(x,0)$ for all $t\in [0,1]$ and $x$ in a stratum $S\subset X$ with $\dim S <n$.
\end{enumerate}
This is an equivalence relation with the property that $f\sim g$ implies $f \circ h \sim g\circ h$ and $h\circ f \sim h \circ g$ for stratified $h$. The first implication uses the fact that a stratified map sends strata to strata of equal or lower dimension.

\begin{definition}
\label{stratn}
Fix $n \in \N \cup \{\infty\}$. Let $\strat{n}$ be the category whose objects are the compact cellular stratified spaces of dimension $\leq n$ and whose morphisms are homotopy classes of stratified maps relative to strata of dimension $<n$. In particular $\strat{\infty}$ is the category of stratified spaces and stratified maps between them. The category $\strat{n}$ is small; the objects are certain subsets of Euclidean spaces, and the morphisms certain maps between these subsets. 
\end{definition}

\subsection{The stratified site}
\label{topology}

In this section we specify a Grothendieck topology on $\strat{n}$ so that it becomes a site. Recall that to do so we must specify a collection of covering sieves for each object $X$, satisfying certain conditions. A sieve on $X$ is a collection of morphisms with target $X$ which is closed under precomposition. First we need the following lemma.
\begin{lemma}
\label{fibre prod}
Suppose $f: X \to Z \leftarrow Y:g$ are stratified maps. Then $X \times_Z Y$ can be stratified by the fibre products of the strata of $X$ and $Y$ so that 
\begin{equation}
\label{fibre prod diag}
\xymatrix{
X\times_Z Y \ar[d] \ar[r] & Y \ar[d]^g\\
X \ar[r]_f & Z
}
\end{equation}
is a commuting diagram of stratified maps. Moreover, if $X, Y$ and $Z$ are cellular then this stratification is cellular.
 \end{lemma}
 \begin{proof}
Consider $X\times_Z Y = \{ (x,y) \in X \times Y \ : \ f(x)=g(y) \} \subset X\times Y$. We equip this with the germ along this subset of the product of the germs of $X$ and $Y$. It is decomposed into the subsets $A \times_{f(A)=g(B)} B$ where $A\subset X$ and $B \subset Y$ are strata. Each of these is a manifold because $f|_A$ and $g|_B$ are transversal. This decomposition satisfies the Whitney B condition: Suppose $(a_i, b_i) \in A\times_ZB$ and $(s_i,t_i)\in S\times_ZT$ are sequences in $X\times_ZY$ with the same limit $(a,b)\in A\times_Z B$. The product stratification of $X\times Y$ satisfies the Whitney B condition. Hence (when the limits exist in the ambient tangent space) 
$$
\lim \ell_i \in \lim T_{(s_i,t_i)}(S\times T)
$$
where $\ell_i$ is the secant line between $(a_i,b_i)$ and $(s_i,t_i)$. In fact since these pairs lie in the fibre product, the limiting secant line lies in the subspace
$$
U=\{ (v,w) \in \lim T_{(s_i,t_i)}(S\times T) \ : \ df(v)=dg(w) \}.  
$$
Clearly $U \supset \lim T_{(s_i,t_i)}(S\times_Z T)$; in fact they are equal. For suppose $(v_i,w_i)\in T_{(s_i,t_i)} (S\times T)$ is a sequence with limit $(v,w)$. Then 
$$
df(v_i)-dg(w_i) \to df(v) - dg(w) = 0.
$$
Since $f$ is submersive onto the tangent space of $f(S)=g(T)$ we can find $v'_i\in T_{s_i}S$ with $df(v'_i) = df(v_i)-dg(w_i)$ and $v'_i\to 0$. Then 
$$
(v_i-v'_i,w_i)\in T_{(s_i,t_i)}(S\times_Z T)
$$
and $(v_i-v'_i,w_i) \to (v,w)$. Hence $U\subset  \lim T_{(s_i,t_i)}(S\times_Z T)$ as claimed. Therefore the given decomposition of the fibre product satisfies the Whitney B condition, and the fibre product becomes a stratified space. It is easy to check that the maps in (\ref{fibre prod diag}) are stratified. 

Suppose that $X, Y$ and $Z$ are cellular. Then by considering the long exact sequences of homotopy groups induced respectively from the fibrations $F \hookrightarrow T \to g(T)$ and $F \hookrightarrow S\times_Z T \to S$ and using the fact that each of $S, T$ and $g(T)$ is contractible we see that  $S\times_ZT$ is weakly contractible. Since it is a smooth manifold it is homotopy equivalent to a CW complex, and so by Whitehead's Theorem it is contractible. Hence $X\times_ZY$ is cellular.
\end{proof}
\begin{remark}
The stratified space $X\times_ZY$ is not in general a fibre product in $\strat{\infty}$ (because of the constraints on dimension there is no hope that $\strat{n}$ for $n\in \N$ will have products). For example if $Z=\pt$ and $X=Y=\cube{}$ then $X^2$ does not have the required universal property because the inclusion of the diagonal is weakly stratified but not stratified. Moreover, it is impossible to subdivide the stratification of $X^2$ so that it becomes a fibre product; to do so one would require that the graph of every strictly monotonic and surjective function $(0,1)\to (0,1)$ was a stratum. Hence we have the stronger statement that the category of stratified spaces and maps does not have products in general.

Despite not being a fibre product, many familiar properties hold, in particular there is an isomorphism $W\times_X (X \times_Y Z) \cong W \times_Y Z$ in $\strat{\infty}$.
\end{remark}

\begin{definition}
\label{covering sieve}
A stratified map $f:Y \to X$ \defn{trivially covers} a stratum $A\subset X$ if $f^{-1}A$ is a single stratum and $f|_{f^{-1}A}$ a diffeomorphism. 
\end{definition}
\begin{proposition}
There is a Grothendieck topology on $\strat{n}$ in which a \defn{covering sieve} on $X$ is one such that for each stratum of $X$ there is a map in the sieve trivially covering that stratum. (In general a covering sieve will contain many such maps.)
\end{proposition}
\begin{proof}
We verify the axioms for a topology. Suppose $\sieve{S}$ is a covering sieve on $X$ and $g: X' \to X$ any stratified map. Then the pullback sieve 
$$
g^*\sieve{S} = \{ f': Y' \to X' \ | \ gf'\in\sieve{S} \}
$$
 should also be a covering sieve. Fix a stratum $A'\subset X'$. Suppose that $f:Y \to X$ trivially covers the stratum $g(A)$. By Lemma \ref{fibre prod} there is a commutative diagram
$$
\xymatrix{
A' \times_X f^{-1}\left(g(A')\right) \ar[r] \ar[d] & X'\times_X Y \ar[r] \ar[d]^{f'} & Y \ar[d]^f\\
A' \ar[r] &  X' \ar[r]_g & X.
}
$$ 
of stratified maps. The pre-image of $A'$ under $f'$ is the single stratum $A' \times_X f^{-1}\left(g(A')\right)$ and the left hand vertical map is a diffeomorphism. The closure of the latter stratum is of dimension $\leq n$, therefore is in the pullback sieve and trivially covers $A'$. Hence the pullback sieve is a covering sieve for $X'$.

Let $\sieve{S}$ be a covering sieve on $X$, and let $\sieve{T }$ be any sieve on $X$. Suppose that for each stratified map $f: Y \to X$ in $\sieve{S}$, the pullback sieve $f^*\sieve{T}$ is a covering sieve on $Y$. We must show that $\sieve{T}$ is a covering sieve on $X$. Fix a stratum $A\subset X$. Since $\sieve{S}$ is a covering sieve for $X$ we can find $f: Y \to X$ trivially covering $A$. Since $f^*\sieve{T}$ is a covering sieve for $Y$ we can find $g: Z \to Y$ trivially covering $f^{-1}A$. Then $gf: Z \to X$ is in the sieve $\sieve{S}$ and trivially covers $A$, so we are done.

Finally, we must verify that the maximal sieve of all stratified maps with target $X$ is a covering sieve. This is immediate since the identity map is in the maximal sieve and trivially covers every stratum.
\end{proof}
Having defined a topology we may speak of sheaves on $\strat{n}$. Recall that these are presheaves $\sh{A}$ such that elements of $\sh{A}(X)$ are given by matching families of elements for any covering sieve. More precisely, consider a covering sieve $\sieve{S}$ on $X$ as a presheaf 
$Y \mapsto \{ f: Y \to X \ |  \ f\in \sieve{S} \}$.
Then a presheaf $\sh{A}$ is a sheaf if and only if the map
\begin{equation}
\label{sheaf condition}
\sh{A}(X) \to \nat{\sieve{S}}{\sh{A}} : a \mapsto \left( f \mapsto f^*a\right)
\end{equation}
is an isomorphism for each covering sieve $\sieve{S}$. A natural transformation $\eta \in \nat{\sieve{S}}{\sh{A}}$ is a collection of elements $a_f \in \sh{A}(Y)$ for each $f:Y\to X$ in the sieve $\sieve{S}$ which `match' in the sense that $g^*a_f = a_{fg}$ for any $g: Y' \to Y$. Here, and in the sequel, we write $g^*$ for $\sh{A}(g)$. In these terms, $\sh{A}$ is a sheaf if and only if each matching family has a unique amalgamation $a\in \sh{A}(X)$ such that $a_f = f^*a$.

\subsection{Prestratified maps}

Stratified maps are rather rigid, and a more flexible notion is useful. A \defn{subdivision} $X'$ of a stratified space $X$ is a Whitney stratification of the underlying space of $X$ each of whose strata is contained within some stratum of $X$. We equip $X'$ with the same stable germ.
\begin{definition}
A smooth map $f: X \to Y$ is \defn{prestratified} if it becomes stratified with respect to some subdivision of the source $X$. 
\end{definition}
Clearly any stratified map is prestratified. If $X'$ is a non-trivial subdivision of $X$ then the identity $X \to X'$ is prestratified, but not {\it vice versa}.

\begin{lemma}
If $f: X \to Y$ and $g:Y\to Z$ are prestratified then so is the composite $gf: X \to Z$. 
\end{lemma}
\begin{proof}
Choose subdivisions $X'$ of $X$ and $Y'$ of $Y$ so that $f:X'\to Y$ and $g:Y'\to Z$ are stratified. Suppose $A$ is a stratum of $X'$. Then $f(A)$ is a stratum of $Y$. Further suppose $B$ is a stratum of $Y'$ contained in $f(A)$. Then $f|_A: A \to f(A)$ is a submersion and hence is transversal to $B$. So $f^{-1}(B) \cap A$ is a submanifold of $A$. The collection of these as $A$ and $B$ vary through the strata of $X'$ and $Y'$ respectively forms a decomposition $X''$ of $X$, subordinate to the stratification $X'$. The Whitney conditions for $X''$ follow from those for $X'$ and for $Y'$. To see this recall that we need only verify the Whitney B condition. Suppose $x_i \in f^{-1}B_0 \cap A_0$ and $y_i \in f^{-1}B_1 \cap A_1$ are sequences with common limit $x\in f^{-1}B_0 \cap A_0$. When the limiting secant and tangent plane exist, 
$$
\lim \overline{x_iy_i} \in \lim T_{y_i}A_1
$$
by Whitney B for $X'$. Now consider the image sequences $f(x_i) \in T$ and $f(y_i) \in T'$. By Whitney B for $Y'$ we know that
$$
df\left( \lim \overline{x_iy_i}\right) = \lim \overline{ f(x_i)f(y_i)} \in \lim T_{f(y_i)}B_1.
$$
Combining these we see that $\lim \overline{x_iy_i} \in \lim T_{y_i} ( f^{-1}B_1 \cap A_1)$ as required.  By construction the composite $gf:X'' \to Z$ is stratified.
\end{proof}
\begin{definition}
Fix $n \in \N \cup \{\infty\}$. Let $\prestrat{n}$ be the category whose objects are the compact cellular stratified spaces of dimension $\leq n$ and whose morphisms are homotopy classes of prestratified maps relative to strata of dimension $<n$. The definition of homotopy used here is identical to that just prior to Definition \ref{stratn}, except that we replace stratified by prestratified throughout. Like $\strat{n}$ this is a small category.
\end{definition}

\section{Whitney categories}
\label{whitney cat}
\begin{definition}
\label{wcat defn}
Fix $n\in \N \cup \{\infty\}$. A \defn{Whitney $n$-category} $\wcat{A}$ is a presheaf of sets on $\prestrat{n}$ such that the restriction to $\strat{n}$ is a sheaf. A functor between Whitney $n$-categories is a map of presheaves, \ie a natural transformation. Whitney $n$-categories and functors between them form a full subcategory $\whit{n}$ of the presheaves.
\end{definition}
We refer to the elements of $\wcat{A}(X)$ as the \defn{morphisms of shape $X$} or as \defn{$X$-morphisms} of $\wcat{A}$. We also refer to the elements of the set $\wcat{A}(\pt)$ associated to a point  as the objects of $\wcat{A}$.

A Whitney $0$-category is completely determined by the set $\wcat{A}(\pt)$. More precisely, the functor $\wcat{A} \mapsto \wcat{A}(\pt)$ from the category of Whitney $0$-categories to the category of sets is an equivalence. In \S\ref{relation to ordinary categories} we indicate why the category of Whitney $1$-categories is equivalent to the category of small dagger categories and dagger functors. Full details will appear in \cite{Smyth:2012fk}, in which the case $n=2$ is  also treated; here there is an equivalence between (the categories of) one-object Whitney $2$-categories and dagger rigid monoidal categories.

\begin{lemma}
\label{whit is complete}
The category $\whit{n}$ is complete.
\end{lemma}
\begin{proof}
Limits are computed object-wise, \ie we set
$$
\left(\lim_i \wcat{A}_i\right)(X) = \lim_i \left( \wcat{A}_i(X) \right)
$$
where the right hand limit is computed in $\sets$. The result is a presheaf, indeed it is the limit in the category of presheaves. Using the fact that $\whit{n}$ is a full subcategory, it suffices to show that $\lim_i \wcat{A}_i$ is in fact a Whitney category. That it is follows from the fact that categories of sheaves are complete with the limits being computed object-wise as above.
\end{proof}
We note some consequences. Firstly $\whit{n}$ has fibre products. Secondly $\whit{n}$ is a monoidal category under the cartesian product of Whitney $n$-categories. Completeness is also key to the next result.

\begin{theorem}
The inclusion $\whit{n} \hookrightarrow \presh{\prestrat{n}}$ has a left adjoint $\whitify$. We refer to $\whitify(\sh{A})$ as the \defn{Whitney category associated to $\sh{A}$}. 
\end{theorem}
\begin{proof}
We use the adjoint functor theorem. Recall that this guarantees the existence of the claimed left adjoint if
\begin{enumerate}
\item $\whit{n}$ is complete;
\item the inclusion $\whit{n} \hookrightarrow \presh{\prestrat{n}}$ is continuous;
\item for each $\sh{A} \in \presh{\prestrat{n}}$ there is a collection $f_i : \sh{A} \to \wcat{B}_i$ of morphisms to Whitney $n$-categories, indexed by a {\em set} $I$, such that any morphism $\sh{A} \to \wcat{B}$ factors through some $f_i$.
\end{enumerate}
The first two conditions follow from Lemma \ref{whit is complete} above. It remains to verify the third condition. To do so we show that the collection of quotient presheaves of $\sh{A}$ forms a set. The required $\{f_i\}$ can then be taken to be the subset of these quotients whose target is a Whitney category. 

A quotient $\sh{A} \to \sh{Q}$ of the presheaf is determined by a (compatible) collection of surjections $\sh{A}(X) \to \sh{Q}(X)$ for each $X \in \prestrat{n}$. The maps $\sh{Q}(Y) \to \sh{Q}(X)$ in the quotient presheaf are completely determined by the corresponding maps in $\sh{A}$. Such surjections are indexed by equivalence relations on the set $\sh{A}(X)$, which we think of as subsets of $\sh{A}(X)^2$. So quotients of $\sh{A}$ can be indexed by a subset of the product of power sets
$$
\prod_{X \in \prestrat{n}} 2^{ \sh{A}(X)^2 }
$$ 
(which exists as a set because $\prestrat{n}$ is small). 
\end{proof}

\begin{remark}
It would be useful to have an actual construction of the left adjoint, perhaps using a modified version of the double plus construction for sheafification. However, the construction of $\whitify$ cannot be exactly like the latter because, in contrast to the plus construction, $\whitify$ cannot preserve finite limits. If it did  then it would follow that $\whit{n}$ was a topos, but this is not the case. For instance we will show in \S\ref{relation to ordinary categories} that $\whit{1}$ is equivalent to the category of dagger categories and functors, and the latter is not a topos (it  has no subobject classifier).
\end{remark}

\begin{corollary}
The category $\whit{n}$ is cocomplete.
\end{corollary}
\begin{proof}
Recall that categories of presheaves are cocomplete (colimits are computed object-wise). It follows that
$$
\colim_i \wcat{A}_i \cong \whitify\left( \colim_i \wcat{A}_i \right)
$$
where the left hand colimit is computed in $\whit{n}$ and the right hand one in $\presh{\prestrat{n}}$.
\end{proof}

\begin{proposition}
\label{loops etc}
Let $\wcat{A}$ be a Whitney $n$-category and $P \in \strat{n}$ with $\dim P =p$. Then the assignments $X \mapsto \wcat{A}(P \times X)$ and 
$$
f \mapsto (\id\times f)^* : \wcat{A}(P\times X) \to \wcat{A}(P\times Y)
$$
define a Whitney $(n-p)$-category which we denote $\wcat{A}^P$.
\end{proposition}
\begin{proof}
It is clear that $\wcat{A}^P$ is a presheaf on $\prestrat{n-p}$, so we need only check that it restricts to a sheaf on $\strat{n-p}$. Let $\{ f_i: X_i \to X \}_{i \in I}$ be a covering sieve for $X$ in $\strat{n-p}$. Then 
\begin{equation}
\label{product family}
\{ f_i \times \id : P \times X_i \to P \times X \}_{i\in I}
\end{equation}
generates a covering sieve for $P \times X$ in $\strat{n}$ whose elements are the stratified maps to $P \times X$ factoring through one of these. The presheaf $\wcat{A}^P$ is a Whitney category if each matching family for (\ref{product family}) extends to a matching family for the sieve which it generates. In other words we must check that whenever we have a commuting diagram
$$
\xymatrix{
W \ar[d]_{g_j} \ar[r]^{g_i} & P\times X_i \ar[d]^{\id \times f_i}\\
P\times X_j \ar[r]_{\id \times f_j} & P \times X
}
$$
and a matching family  $\{a_i \in \wcat{A}(P\times X_i) \}_{i\in I}$ for the maps in (\ref{product family}) that $g_i^*a_i = g_j^*a_j$. Since $\wcat{A}$ is a Whitney category it suffices to show that $ h^*g_i^*a_i = h^*g_j^*a_j$ for all $h$ in some covering sieve of $W$. We construct a covering sieve with this property as follows. For each stratum $S_k \subset W$ there is an image stratum in $P\times X$ and a map  $\id \times f_k :  P \times X_k \to P \times X$ in (\ref{product family}) trivially covering it. Consider the commuting diagram
$$
\xymatrix{
W_k \ar[rr] \ar[dd] \ar[dr]^{h_k} && P\times (X_i \times_X X_k)\ar'[d][dd] \ar[dr] & \\
& W \ar[rr]^<<<<<<<<{g_i} \ar[dd]_<<<<<<<<{g_j} && P \times X_i \ar[dd]^{\id \times f_i} \\
P\times (X_j \times_X X_k) \ar'[r][rr] \ar[dr] && P\times X_k \ar[dr]^{\id \times f_k} &\\
& P \times X_j \ar[rr]_{\id \times f_j} && P \times X
}
$$
where we set $W_k = (P \times X_k)\times_{P\times X} W $ and stratify the fibre products as in Lemma \ref{fibre prod}. (To be precise, for the diagram to exist in $\strat{n}$ we must expunge any strata of dimension $>n$. But this does not effect the argument.) By construction $h_k: W_k\to W$ trivially covers the stratum $S_k$. The collection of the $h_k$ for all strata $S_k$ in $W$ thus generates a covering sieve of $W$, namely all those maps to $W$ which factor through one of the $h_k$. Since we have a matching family for the maps in (\ref{product family}) it follows from the diagram that both $h_k^*g_i^*a_i$ and $h_k^*g_j^*a_j$ agree with the pullback of $a_k$ from $\wcat{A}(P \times X_k)$, and so they are equal. Therefore they agree on the covering sieve for $W$ constructed above, and so $g_i^*a_i = g_j^*a_j$ as required.
\end{proof}
\begin{corollary}
Let $\wcat{A}$ be a Whitney $n$-category and fix objects $a,a' \in \wcat{A}(\pt)$. The assignment
$$
X \mapsto \{ \alpha \in \sh{A}(X \times \cube{}) \ : \ \imath_0^*\alpha = p^*a, \imath_1^*\alpha = p^*a' \},
$$
where $\imath_t : X \times t \hookrightarrow X \times \cube{}$ is the inclusion and $p:X \to \pt$ the map to a point, defines a Whitney $(n-1)$-category $\wcat{A}(a,a')$. 
\end{corollary}
\begin{proof}
This is a special case of the proof of Proposition \ref{loops etc} above, with $P=\cube{}$, except that we now have boundary conditions. That is we are working with the sub-presheaf of $\wcat{A}^{\cube{}}$ consisting of elements $\alpha$ satisfying $\imath_0^*\alpha = p^*a$ and $\imath_1^*\alpha = p^*b$. Since the amalgamation of a matching family of elements with this property also has this property the argument goes through as before.
\end{proof}
We refer to $\wcat{A}(a,a')$ as the \defn{Whitney category of morphisms} from $a$ to $a'$. The construction is functorial: given $F : \wcat{A} \to \wcat{B}$ and $a,a'\in  \wcat{A}(\pt)$ there is an induced morphism 
$$
F(a,a') : \wcat{A}(a,a') \to \wcat{B}(Fa,Fa')
$$ 
of Whitney $(n-1)$-categories.

In order to compare Whitney $n$-categories we need a notion of equivalence. We model it on the following symmetric version of equivalence of ordinary small categories: an equivalence of $\cat{A}$ and $\cat{B}$ is given by a span 
$$
\xymatrix{
\wcat{A}  & \wcat{C} \ar[l]_F \ar[r]^G & \wcat{B}
}
$$
in which the functors $F$ and $G$ are fully-faithful and surjective (not merely {\em essentially} surjective) on objects. We use this to make the following inductive definition.
\begin{definition}
\label{equivalent}
For $n >0$ an \defn{$n$-equivalence} of Whitney $n$-categories is a span
$$
\xymatrix{
\wcat{A}& \wcat{C} \ar[l]_F \ar[r]^G & \wcat{B}
}
$$
of functors which are surjective on objects, \ie the maps $F(\pt): \wcat{C}(\pt) \to \wcat{A}(\pt)$ and $G(\pt):\wcat{C}(\pt) \to \wcat{B}(\pt)$ are surjective, and which induce $(n-1)$-equivalences
$$
\xymatrix{
\wcat{A}(Fc,Fc')  & \wcat{C}(c,c') \ar[l]_{\quad F} \ar[r]^{G\quad } & \wcat{B}(Gc,Gc').
}
$$
for each pair $c,c'\in \wcat{C}(\pt)$. A $0$-equivalence is a span such that $F(\pt)$ and $G(\pt)$ are bijections. 
\end{definition}
\begin{proposition}
The notion of $n$-equivalence is an equivalence relation on Whitney $n$-categories. 
\end{proposition}
\begin{proof}
Reflexivity and symmetry are immediate. Transitivity follows from the fact that we can compose spans using the fibre product:
$$
\xymatrix{ 
&& \wcat{D}\times_\wcat{B} \wcat{E} \ar@{-->}[dl] \ar@{-->}[dr] \\
& \wcat{D} \ar[dl] \ar[dr] && \wcat{E} \ar[dl] \ar[dr]  \\
\wcat{A} && \wcat{B} && \wcat{C}.
}
$$
We claim that the outer roof is an equivalence whenever the inner roofs are equivalences. We use induction on $n$. The base case $n=0$ is clear. Assume the result holds for $(n-1)$-equivalences. Suppose we are given a diagram as above. Evaluating at a point the solid arrows are, by assumption, surjective. Hence so are the dotted ones. Therefore the induced maps $(\wcat{D}\times_\wcat{B} \wcat{E})(\pt) \to \wcat{A}(\pt), \wcat{C}(\pt)$ are surjective. For any objects $(d,e)$ and $(d',e')$ in $(\wcat{D}\times_\wcat{B} \wcat{E})(\pt)$ there is an induced diagram of morphism categories. Using the fact that
$$
(\wcat{D}\times_\wcat{B} \wcat{E})\left( (d,e) , (d',e') \right) \cong \wcat{D}(d,d') \times_{\wcat{B}(b,b')} \wcat{E}(e,e') 
$$
and the inductive hypothesis we deduce that the outer span of the diagram of morphism categories is an $(n-1)$-equivalence. Therefore the outer span of the original diagram is an $n$-equivalence.
\end{proof}

\subsection{Monoidal Whitney categories}
\label{monoidal whitney categories}
Recall that a category with one object is a monoid, that a bicategory with one object is a monoidal category and that a bicategory with one object and one morphism is a commutative monoid. (The latter follows from the fact that if a set has two monoid structures $\star$ and $\ast$ with the distributive property
$$
(a \star b) \ast (c \star d) = (a\ast c) \star (b \ast d)
$$
then $\star$ and $\ast$ agree, and are commutative.) By analogy we define a \defn{$k$-tuply monoidal Whitney $n$-category} to be a Whitney $(n+k)$-category $\wcat{A}$ for which $\wcat{A}(X) = \unit$ is a one element set whenever $\dim X < k$. 

We can obtain a {\it bona fide} Whitney $n$-category from a $k$-tuply monoidal one as follows. Let $\Omega^k\wcat{A}$ be the Whitney $n$-category
$$
Y \mapsto  \{ a\in \wcat{A}\left(\sphere{k} \times Y\right) \ : \ a|_{0 \times Y} = 1 \}
$$
where we write $1$ for the pullback of the unique element in $\wcat{A}(\pt)$ under the unique map to a point. The proof that this is a Whitney category is the special case $P=\sphere{k}$ of Proposition \ref{loops etc}, but with an added `boundary' condition (which does not effect the argument). 

One can define a monoidal structure on $\Omega^k\wcat{A}$  by choosing a prestratified map $\mu: \sphere{k} \to \sphere{k}\vee \sphere{k}$ (the wedge of the spheres identifying the point strata) which is degree one onto each lobe. There is then a unique dotted map such that  
$$
\xymatrix{
\Omega^k\wcat{A}(Y)^2 \ar[d] \ar@{-->}[rr] && \Omega^k\wcat{A}(Y) \ar[d] \\
\wcat{A}(\sphere{k}\times Y) \times_{\wcat{A}(Y)} \wcat{A}(\sphere{k}\times Y) \ar@{=}[r]  & \wcat{A}\left((\sphere{k}\vee \sphere{k})\times Y \right) \ar[r]_{\ \ \ \ (\mu\times\id)^*} & \wcat{A}(\sphere{k} \times Y)
}
$$
commutes. Uniqueness follows because the vertical maps are injections. The identification in the bottom row comes from the fact that $\wcat{A}$ is a sheaf on $\strat{n}$. This defines a binary operation $\Omega^k\wcat{A} \times \Omega^k\wcat{A} \to \Omega^k\wcat{A}$. The distinguished element $1 \in\Omega^k\wcat{A}(Y)$ acts as a (weak) unit. 

However, we prefer to consider $k$-tuply monoidal $n$-categories as special $(n+k)$-categories rather than $n$-categories with additional structure. This has the virtue that a monoidal functor is then simply a functor between $(n+k)$-categories, rather than a functor obeying an extra condition. 

\subsection{Examples}
\label{examples}

In this section we discuss three different classes of  examples of Whitney $n$-categories. 

\subsubsection{Representable Whitney categories}
\label{representables}

For any stratified space $X$ there is a representable Whitney $n$-category $\rep{X}$ given by the presheaf
$$
\rep{X} = \prestrat{n}( -, X)
$$
The only thing to verify is that this restricts to a sheaf on $\strat{n}$. Suppose $\sieve{S}$ is a covering sieve on $Y$. The canonical map
$$
\rep{X}(Y) \to \nat{\sieve{S}}{\rep{X}} : f \mapsto ( g \mapsto f\circ g)
$$
is injective: If the classes of $f, f': Y \to X$ differ as elements of $\rep{X}(Y)$ then their restrictions to some stratum $A\subset Y$ differ, and choosing  $g: Z \to Y$ in the sieve trivially covering $A$ the classes of  $f\circ g$ and $f' \circ g$ differ. 

The canonical map is also surjective: Fix an element of $\nat{\sieve{S}}{\rep{X}}$, \ie a compatible family $\{f_g\}$ of prestratified maps $f_g: Z \to X$ for each $g: Z \to Y$ in the sieve $\sieve{S}$. For each stratum $A\subset Y$ we can define a prestratified map $f_A: A \to X$ by choosing $g: Z \to Y$ trivially covering $A$ and considering the composite
$$
f_A = f_g \circ \left( g|_{g^{-1}A} \right)^{-1}  : A \to g^{-1}A \to X.
$$
Compatibly of the family implies that $f_A$ is independent of the choice of such $g$. Together with the fact that we work with germs of smooth maps it also means that the $f_A$ patch together to form a smooth map $f: Y \to X$. Moreover, $f$ is prestratified because this condition can be checked stratum-by-stratum. Surjectivity follows since, by construction, $f \mapsto \{ f_g\}$.
\begin{remark}
\label{colimit of representables}
Any presheaf is a colimit of representable presheaves, in fact has a canonical such description, see for example \cite[p40]{MR1300636}. Since the left adjoint $\whitify$ to the inclusion $\whit{n} \hookrightarrow \presh{\prestrat{n}}$ preserves colimits it follows that any Whitney category is a colimit of representable ones, again in a canonical way.
\end{remark}

\subsubsection{Framed tangles}
\label{tangles}

Given  $k, n\in \N$ we define a Whitney $(n+k)$-category of framed tangles by setting
$$
\tang{n}{k}(X) = \left\{ 
\begin{array}{c} 
\textrm{Germs of codimension $k$ framed submanifolds } \\
\textrm{$T \subset X$ transverse to each stratum} 
\end{array}
\right\} \big/ \sim.
$$
The equivalence relation $\sim$ is given by ambient isotopy, relative to all strata of dimension strictly less than $n+k$. By a framing we mean a homotopy class of trivialisations of the normal bundle, where the homotopy is the identity over intersections of the tangle with strata of dimension $<n+k$. In particular if $\dim X < n+k$ then a framing is a fixed trivialisation. 
\begin{remark}
In the `classical' case of $1$-dimensional tangles in $3$-dimensional space, corresponding to $n=1$ and $k=2$, the adjective framed is more commonly used in the knot theorists' sense of a chosen non-vanishing section of the normal bundle. In these terms, what we call a framed tangle would be instead a framed and oriented tangle. Despite the unfortunate clash of terminology, we use the topologists' notion of framing since it generalises appropriately to higher dimensions.
\end{remark}

To complete the definition we need to specify the map induced by prestratified $f: X \to Y$. We define
$$
f^* : \tang{n}{k}(Y)  \to \tang{n}{k}(X) : T \mapsto f^{-1}T.
$$
Since $f$ is prestratified and $T$ transversal to all strata of $Y$ the pre-image $f^{-1}T$ is a submanifold of codimension $k$ in $X$, also transversal to all strata. It inherits a framing given by the isomorphisms 
$$
N\left(f^{-1}T\right) \cong f^*NT \cong f^*\left(T \times \R^k\right) \cong f^{-1}T \times \R^k.
$$
Homotopic maps give rise to isotopic framed submanifolds, hence $f^*$ is well-defined. The verification that this restricts to a sheaf is similar to the case of representable presheaves. A matching family of germs of framed submanifolds of each stratum  amalgamates to form a germ of a framed submanifold of the entire stratified space.

Note that $\tang{n}{k}$ is a $k$-tuply monoidal $n$-category, since only the empty codimension $k$ submanifold is transversal to the strata of a space $X$ with $\dim X < k$.

As an example, in \S\ref{relation to ordinary categories}, we explain how to recover a more familiar version of the category of framed tangles in the case $n=1$ and $k=2$.

\subsubsection{Transversal homotopy theory}
\label{transversal homotopy theory}
Let $M$ be a Whitney stratified manifold with a generic basepoint $p$, \ie $p$ lies in some open stratum of $M$. We define `transversal homotopy Whitney categories' of $M$ built out of maps into $M$ which are transversal to all strata. To be precise a smooth map $g:X\to M$ from a stratified space into $M$ is \defn{transversal to all strata of $M$} if for each stratum of $S\subset X$ the restriction $g|_S$ is transversal to the inclusion of each stratum of $M$. 

For each $k,n \in \N$ we associate a Whitney $(n+k)$-category $\thcat{k,n+k}{M}$ to $M$ by defining
$$
\thcat{k,n+k}{M}(X) = 
\left\{ 
\begin{array}{c} 
\textrm{Transversal $g: X \to M$ such that whenever } \\
\textrm{$S \subset X$ and $\dim S < k$ then $S \subset g^{-1}(p)$}
\end{array}
\right\} \big/ \sim.
$$
Here $\sim$ is the equivalence relation given by homotopy through transversal maps relative to all strata $S\subset X$ with $\dim S < n+k$. Write $[g]$ for the class of $g: X \to M$.

Given prestratified $f: X \to Y$ we define $f^*[g] = [g \circ f]$.  Then $g\circ f$ is transversal to all strata of $M$ and $[g\circ f]$ depends only on the morphism in $\prestrat{n+k}$ represented by $f$. The verification that this restricts to a sheaf on $\strat{n+k}$ is similar to that for representable presheaves. The condition that $g(S)=p$ whenever $\dim S < k$ means that this is a $k$-tuply monoidal Whitney $n$-category.

Transversal homotopy Whitney categories are functorial for sufficiently nice maps  between Whitney stratified manifolds. Specifically, they are functorial for  \defn{weakly stratified normal submersions} $h:M\to N$, \ie weakly stratified maps such that the induced mappings  $N_pS \to N_{h(p)}h(S)$ of normal spaces to strata are always surjective.\footnote{Such maps were termed `stratified normal submersions' in \cite{MR2720181} where the notion of stratified map was weaker.} Whenever  $h:M \to N$ is a weakly stratified normal submersion and $g: X \to M$ is transversal then the composite $h \circ g : X \to N$  is transversal. So we can define a map
$$
\thcat{k,n+k}{M}(X) \to \thcat{k,n+k}{N}(X) :  [g] \mapsto [h\circ g].
$$
Since composition on the left and right commute this specifies a natural transformation of presheaves, \ie  a functor $\thcat{k,n+k}{M} \to \thcat{k,n+k}{N}$ of Whitney categories. 

In the case $n=0$ one can recover the transversal homotopy monoids $\psi_k(M)$ defined in \cite{MR2720181} by considering the associated Whitney $0$-category $\Omega^k\thcat{k,k}{M}$. This is completely determined by its set of objects
$$
\Omega^k\thcat{k,k}{M}(\pt) = \thcat{k,k}{M}(\sphere{k}),
$$
\ie  the set of homotopy classes of based transversal maps $\sphere{k} \to M$. This is the underlying set of the dagger monoid $\psi_k(M)$. The monoid structure can be recovered by the procedure outlined in \S\ref{monoidal whitney categories}, and the dagger structure from the map induced by a reflection of $\sphere{k}$. In \S\ref{relation to ordinary categories} we will sketch the analogous relation for $n=1$ to the transversal homotopy categories defined in \cite{MR2720181}.

\subsection{Relation to `ordinary' categories}
\label{relation to ordinary categories}

On the face of it the definition of Whitney $n$-category seems rather remote from `ordinary' category theory. In order for our definition of Whitney $n$-category to be a reasonable notion of `$n$-category with duals' it should agree with the accepted definitions for small $n$. The case $n=0$ is rather trivial, as $0$-categories with duals and Whitney $0$-categories are both simply sets. In this section we discuss the more interesting $n=1$ case. We sketch constructions producing a small dagger category from a Whitney $1$-category and {\it vice versa}. These are functorial and induce equivalences between the category $\whit{1}$ of Whitney $1$-categories and the category of small dagger categories and dagger functors. Full details  will appear in \cite{Smyth:2012fk}. The case $n=2$ in which there is a close relation between one-object Whitney $2$-categories and rigid dagger categories will also be addressed there.

\subsubsection{Whitney $1$-categories}
\label{n=1}

Let $\dagcat$ be the category of small dagger categories and dagger functors between them, \ie functors which commute with the dagger duals.
\begin{theorem}
There are functors 
$$
\xymatrix{
\whit{1} \ar@/^1pc/[r]^{\ \mathcal{D}} &  
\dagcat \ar@/^1pc/[l]^{\ \mathcal{W}}
}
$$
giving an equivalence of categories. 
\end{theorem}
\begin{proof}[Sketch proof]
Given a Whitney $1$-category $\cat{A}$ we define a small dagger category $\mathcal{D}(\wcat{A})$ with objects $\wcat{A}(\pt)$ and morphisms $\wcat{A}([0,1])$. Source and target maps are induced from the inclusions $\imath_0$ and $\imath_1$ of $0$ and $1$ respectively into $[0,1]$. Identities on objects arise from the map induced from $p:[0,1] \to \pt$. Composition is induced from $c:[0,1] \to [0,2] : t \mapsto 2t$, where we stratify by the integer points and their complement. Note that the sheaf condition implies that 
$$
\wcat{A}([0,2]) \cong  \wcat{A}([0,1]) \times_{\wcat{A}(\pt)} \wcat{A}([0,1]) 
$$
is the set of composable pairs of morphisms. The dagger dual is induced from $d: [0,1] \to [0,1] : t \mapsto 1-t$. The various equations --- associativity of composition, triviality of composition with an identity, the equations satisfied by the dagger dual and so on --- arise from homotopies between prestratified maps. 

Given a functor between Whitney $1$-categories, in other words  a natural transformation $\eta: \wcat{A} \to \wcat{B}$, we define a functor $\mathcal{D}(\eta) : \mathcal{D}(\wcat{A}) \to \mathcal{D}(\wcat{B})$ by $a \mapsto \eta_\pt(a)$ on objects and $\alpha \mapsto \eta_{[0,1]}(\alpha)$ on morphisms. That this is  a dagger functor follows from the naturality square
$$
\xymatrix{
\wcat{A}([0,1]) \ar[r]^{\eta_{[0,1]}} \ar[d]_{d^*} & \wcat{B}([0,1]) \ar[d]^{d^*} \\
\wcat{A}([0,1]) \ar[r]_{\eta_{[0,1]}}& \wcat{B}([0,1]).
}
$$ 
The above constructions define a functor $\mathcal{D}: \whit{1} \to \dagcat$.

In the other direction, suppose $\cat{D}$ is a dagger category. We define a Whitney $1$-category $\mathcal{W}(\cat{D})$ by associating to a stratified space $X$ a set of equivalence classes of labellings of $X$ by objects and morphisms in $\cat{D}$. To assign a labelling we
\begin{enumerate}
\item  choose an orientation for each $1$-dimensional stratum of $X$;
\item label each $0$-dimensional stratum by an object of $\cat{D}$;
\item label each (oriented) $1$-dimensional stratum by a morphism of $\cat{D}$ compatibly with the objects labelling the endpoint(s).
\end{enumerate}
Two such labellings are equivalent if they have the same class under the equivalence relation generated by reversing the orientation of a $1$-dimensional stratum and replacing the labelling morphism by its dagger dual.

Given prestratified $f: X \to Y$ we define the map $f^*: \mathcal{W}(\cat{D})(Y) \to  \mathcal{W}(\cat{D})(X)$ by `pulling back' labellings from $Y$ to $X$. More precisely, we label a $0$-dimensional stratum in $X$ by the object labelling its image, necessarily a $0$-dimensional stratum, in $Y$. To assign a label to each (oriented) $1$-dimensional stratum in $X$ it in fact suffices to describe how to do so for the maps $p:[0,1] \to \pt$, 
$$
d: [0,1] \to [0,1] : t \mapsto 1-t
$$ and $[0,1] \to [0,n]: t \mapsto nt$. In these cases we assign respectively the identity on the object labelling the image point, the dagger dual of the morphism labelling the $1$-dimensional image stratum and the $n$-fold composite of the morphisms labelling the $1$-dimensional image strata. One can show that  $\mathcal{W}(\cat{D})$ is a presheaf on $\prestrat{1}$. The restriction to $\strat{1}$ is a sheaf, essentially because labellings are `local'.

Given a dagger functor $F:\cat{D} \to \cat{E}$ one can map a $\cat{D}$-labelling of $X$ to an $\cat{E}$-labelling by applying $F$ to each label. When $F$ is a dagger functor this respects the equivalence relation on labellings and yields a natural transformation $\mathcal{W}(F) : \mathcal{W}(\cat{D}) \to \mathcal{W}(\cat{E})$.  We have therefore defined a functor $\mathcal{W}: \dagcat \to \whit{1}$.

These constructions are inverse to one another. There is a natural isomorphism of dagger categories
$\cat{D} \to \mathcal{D} \mathcal{W} \left( \cat{D} \right)$ which is the identity on objects and takes a morphism $f$ to the class of the labelling of $[0,1]$, with standard orientation, and label $f$. In the other direction, consider fixed $X$ and orient the $1$-dimensional strata. For each stratum $S$ there is then a unique-up-to-homotopy characteristic map $\chi_S : [0,1]^{\dim S} \to X$ which is stratified and of degree one. Moreover, the sheaf condition implies that the map
$\wcat{A}(X) \to \mathcal{W}  \mathcal{D} \left( \wcat{A} \right)(X)$ 
taking $a$ to the obvious labelling of $X$ by the $\chi_S^*a$ is an isomorphism. These maps fit together to form a natural isomorphism $\wcat{A} \to \mathcal{W}  \mathcal{D} \left( \wcat{A} \right)$.
\end{proof}

\begin{example}
\label{transversal homotopy relationship}
Let $M$ be a Whitney stratified manifold with generic basepoint $p$, and let $\wcat{A}=\thcat{k,k+1}{M}$ be the transversal homotopy category defined in \S\ref{transversal homotopy theory}. Using the construction in \S \ref{monoidal whitney categories} one obtains a Whitney $1$-category $\Omega^k\wcat{A}$ with
$$
\Omega^k\wcat{A}(Y) = \{ f \in \thcat{k,k+1}{M}(\sphere{k} \times Y)  \ : \ f(0, y) = p \quad \forall y \in Y \} 
$$
The objects of the dagger category $\mathcal{D}\left( \Omega^k\wcat{A}\right)$ are germs of based transversal maps $\sphere{k} \to M$ and the morphisms are homotopy classes of germs of transversal maps $\sphere{k} \times \cube{} \to M$, mapping $0 \times \cube{}$ to the basepoint $p$, relative to the ends $\sphere{k}\times\{0,1\}$.  This is equivalent to the $k$th transversal homotopy category --- confusingly also denoted $\Psi_{k,k+1}(M)$ --- defined in \cite[\S 4]{MR2720181}. The only difference is that here we use germs of maps, whereas in \cite{MR2720181} smoothness of composites was ensured by imposing stronger boundary conditions.
\end{example}
\begin{example}
\label{tangle relationship}
Let $\wcat{A}=\tang{1}{2}$ be the Whitney $3$-category of framed $1$-dimensional tangles in codimension $2$ of \S\ref{tangles}. This is $2$-tuply monoidal and one can extract a Whitney $1$-category $\Omega^2\wcat{A}$, and from that a dagger category $\mathcal{D}\left(\Omega^2\wcat{A}\right)$. The objects of the resulting dagger category are finite sets of framed points in the open stratum of the sphere $\sphere{2}$. The morphisms are isotopy classes, relative to the boundary, of framed $1$-manifolds in $\sphere{2} \times \cube{} - 0 \times \cube{}$ with (possibly empty) boundary in $\sphere{2}\times \{0,1\}$. This is (a version of) the usual category of normally-framed tangles.
\end{example}

\section{The Tangle Hypothesis}
\label{tangle hypothesis}

Consider the Whitney $(n+k)$-category $\rep{\sphere{k}}$ represented by the stratified sphere $\sphere{k}$. It is a $k$-tuply monoidal Whitney $n$-category: if $\dim X < k$ then any prestratified map $X \to \sphere{k}$ must map $X$ to the point stratum, so $\rep{\sphere{k}}(X)$ has exactly one element.  The identity map $\id: \sphere{k} \to \sphere{k}$ determines a distinguished $\sphere{k}$-morphism.
\begin{lemma}
The Whitney $(n+k)$-category $\rep{\sphere{k}}$ is the free $k$-tuply monoidal Whitney $n$-category on one $\sphere{k}$-morphism.
\end{lemma}
\begin{proof}
This follows from the Yoneda lemma. Given a $k$-tuply monoidal Whitney $n$-category $\wcat{A}$ and an $\sphere{k}$-morphism $a\in \wcat{A}(\sphere{k})$ there is a unique functor of Whitney $(n+k)$-categories with
$$
\rep{\sphere{k}}  \to \wcat{A} : [f: X \to \sphere{k} ]  \mapsto f^*a 
$$
mapping the distinguished element $\id_{\sphere{k}}$ to $a$.
\end{proof}

\begin{proposition}
\label{th}
The $k$-tuply monoidal Whitney $n$-categories $\tang{n}{k}$ and $\rep{\sphere{k}}$ are equivalent.
\end{proposition}
\begin{proof}
We use the Pontrjagin--Thom construction. Fix a generic point $p \in \sphere{k}$. If $f: X \to \sphere{k}$ is prestratified then it is transversal to $p$, because it is submersive onto the open stratum whenever $f^{-1}(p)\neq \emptyset$. Thus the pre-image $f^{-1}(p)$ is (a stable germ of) a framed codimension $k$ submanifold of $X$ which is transversal to all strata. If $f$ and $g$ are homotopic relative to strata of dimension $<n+k$ through prestratified maps then the pre-images $f^{-1}(p)$ and $g^{-1}(p)$ are  isotopic relative to strata of dimension $<n+k$. The assignment $[f] \mapsto [f^{-1}(p)]$ determines a functor
$$
F_p: \rep{\sphere{k}} \to \tang{n}{k}.
$$
Conversely, given a codimension $k$ framed submanifold $T$ of $X$, transversal to all strata, we can construct a prestratified `collapse map' $f:X \to \sphere{k}$ so that $T=f^{-1}(p)$, a tubular neighbourhood of $T$ fibres over the open stratum of $\sphere{k}$ and the complement of this neighbourhood maps to the point stratum. It follows that $F_p(X)$ is surjective for any $X$, in particular for $X=\pt$. Hence $F_p$ is an $(n+k)$-equivalence if and only if the induced functor
\begin{equation}
\label{induced equivalence?}
\rep{\sphere{k}}([f],[g]) \to \tang{n}{k}([f^{-1}(p)],[g^{-1}(p)]) 
\end{equation}
is an $(n+k-1)$-equivalence for any $[f],[g] \in \rep{\sphere{k}}(\pt)$. (Of course there is only one prestratified map $\pt \to \sphere{k}$, and for this the preimage of $p$ is empty. However, we wish to make an inductive argument and the boundary conditions will not always be `trivial' in this way, so we do not  make any assumptions about $f$ and $g$ at this point.)

Checking whether we have an $(n+k-1)$-equivalence in (\ref{induced equivalence?}) is very similar to checking whether $F_p$ is an $(n+k)$-equivalence. The difference is that now we consider prestratified maps $\cube{} \times X \to \sphere{k}$ on the one hand and framed tangles in $\cube{} \times X$ on the other, with boundary conditions on $\{0,1\}\times X$  given by the maps $f$ and $g$, and the preimage tangles $f^{-1}(p)$ and $g^{-1}(p)$ respectively. 

The existence of prestratified collapse maps, extending given ones on the boundary, shows that the induced functor is surjective on objects. So we reduce to checking whether it induces an appropriate $(n+k-2)$-equivalence. Proceeding inductively, we reach the base cases. These concern the map of sets $[f] \mapsto [f^{-1}(p)]$ from homotopy classes of transversal maps $f: [0,1]^{n+k} \to \sphere{k}$, where the restriction of $f$ to the boundary is some fixed map, $\varphi$ say, to the set of isotopy classes of framed codimension $k$ tangles in $[0,1]^{n+k}$, where the boundary tangle is $\varphi^{-1}(p)$. We can always construct a prestratified collapse map such a tangle $T$, extending the given map on the boundary. Indeed, the collapse map is unique up to homotopy through prestratified maps. Moreover, given an isotopy $h_t: [0,1]^{n+k} \to [0,1]^{n+k}$ relative to the boundary, and such that $h_t(T)$ is transversal to all strata for each $t\in [0,1]$, we can construct a family of collapse maps for the the tangles $h_t(T)$ yielding a homotopy between a collapse map for $T=h_0(T)$ and a collapse map for $h_1(T)$. Hence in the base case there is a bijection between isotopy classes of framed tangles and homotopy classes of prestratified collapse maps (each with appropriate boundary conditions). It follows that $F_p$ is an $(n+k)$-equivalence.
\end{proof}

\subsection{Transversal Homotopy Categories of Spheres}
\label{transversal homotopy of spheres}

A minor variant of this proof of the Tangle Hypothesis relates categories of framed tangles to the transversal homotopy categories of spheres. More precisely, taking the pre-image of the point stratum $0 \in \sphere{k}$ induces a functor
$$
F: \thcat{k,n+k}{\sphere{k}} \to \tang{n}{k} : [f] \mapsto [f^{-1}(0)].
$$
There are two differences from the functor $F_p$. Firstly the r\^oles of the basepoint and stratum have been switched:  prestratified maps to $\sphere{k}$ are transversal to the generic basepoint $p$ rather than to the stratum $0$. Secondly, prestratified maps are submersive not just at $p$ but onto the entire open stratum, whereas transversal maps to $\sphere{k}$ are only required to be submersive at the point stratum $0$. 
\begin{proposition}
The functor $F$ is an $(n+k)$-equivalence.
\end{proposition}
\begin{proof}
The proof is almost word-for-word the same as that of Proposition \ref{th}, but with $0$ replacing $p$, and with transversal maps to $\sphere{k}$ replacing prestratified maps. The key point is that one can construct transversal collapse maps for framed tangles, and that these are unique up to homotopy through such maps. See \cite[Appendix A]{MR2720181} for details.
\end{proof}

\subsection{Other flavours of tangles}
\label{other flavours}

Thus far we have considered only framed tangles, however there are variants of the Tangle Hypothesis for oriented tangles, unoriented tangles and so on. To make this more precise, fix a subgroup $G \subset O_k$. The most interesting examples come from stable representations $G_* \to O_*$. Then we can define a $k$-tuply monoidal Whitney $n$-category $\gtang{n}{k}$ of codimension $k$ tangles whose normal bundles have structure group reducing to $G$, or $G$-tangles for short. The framed case corresponds to taking $G=1$, at the other extreme $G=O_k$ corresponds to `plain' tangles with no special normal structure. The group $G$ acts on $\sphere{k}$, considered as $\R^k \cup \{\infty\}$, by prestratified maps fixing the point stratum. Hence there is an induced action on $\wcat{A}(\sphere{k})$ for any Whitney category $\wcat{A}$, and one may speak of $G$-invariant $\sphere{k}$-morphisms. 

In these terms we formulate the Tangle Hypothesis for $G$-tangles as saying that  
\begin{quote}
The Whitney category $\gtang{n}{k}$ is equivalent to the free $k$-tuply monoidal Whitney $n$-category on one $G$-invariant $\sphere{k}$-morphism.
\end{quote}
 Unfortunately it is not straightforward to mimic the proof of the framed case. The difficulty is in finding $X$ with the property $\wcat{A}(X) \cong \wcat{A}(\sphere{k})^G$. The naive candidate is $X=\sphere{k}/G$, but since the action is not free one should presumably consider instead the stack $[\sphere{k}/G]$. Thus one is led to enlarging  the category $\prestrat{n}$ to include suitable stratified smooth stacks. Rather than pursue this, we outline an alternative, more elementary, argument. 

Given a $k$-tuply monoidal Whitney $n$-category $\wcat{A}$ and  $G$-invariant $a\in \wcat{A}(\sphere{k})$ we wish to construct a functor $F:\gtang{n}{k} \to \wcat{A}$ which maps the point $G$-tangle in $\sphere{k}$ to $a$, and moreover to show that such a construction is essentially unique. Thus for each $G$-tangle $T\in \gtang{n}{k}(X)$ we must construct an element $F(T) \in \wcat{A}(X)$, in a canonical fashion. Begin by choosing a small disk-bundle neighbourhood of $T$ in $X$, such that the boundary meets only those strata which $T$ does, and meets these transversely. Next choose a cellular decomposition of the submanifold $T$, subdividing the stratification induced from $X$, for instance by choosing a compatible triangulation. Decompose the disk-bundle neighbourhood into product cells $C \times D^k$ where $C$ is a cell in $T$ and $D^k$ the standard $k$-ball. Finally extend this to a cellular decomposition of $X$ subdividing the original stratification, and denote the subdivision by $s:X \to X'$. Define an element of $\wcat{A}(X')$ by giving a matching family for each cell of this decomposition as follows. For product cells $C\times D^k$ in the disk-bundle neighbourhood assign $\pi^*a$ where 
$$
\pi : C \times D^k \to \sphere{k}
$$
is the composite of second projection and collapse of the disk's boundary. This assignment is forced by the requirement that the point $G$-tangle in $\sphere{k}$ maps to $a$, and this gives rise to the uniqueness of $F$. For cells outside the disk-bundle neighbourhood assign the unit (\ie the pullback of the unique element in $\wcat{A}(\pt)$ under the map to a point). The $G$-invariance of $a$ ensures that the elements assigned to cells in the disk-bundle neighbourhood match. Composing via $s^*:\wcat{A}(X') \to \wcat{A}(X)$ yields the required $F(T)$. Figure \ref{tangle hyp} illustrates the construction.
\begin{figure}
\vspace{.2in}
\centerline {
\includegraphics[width=2in]{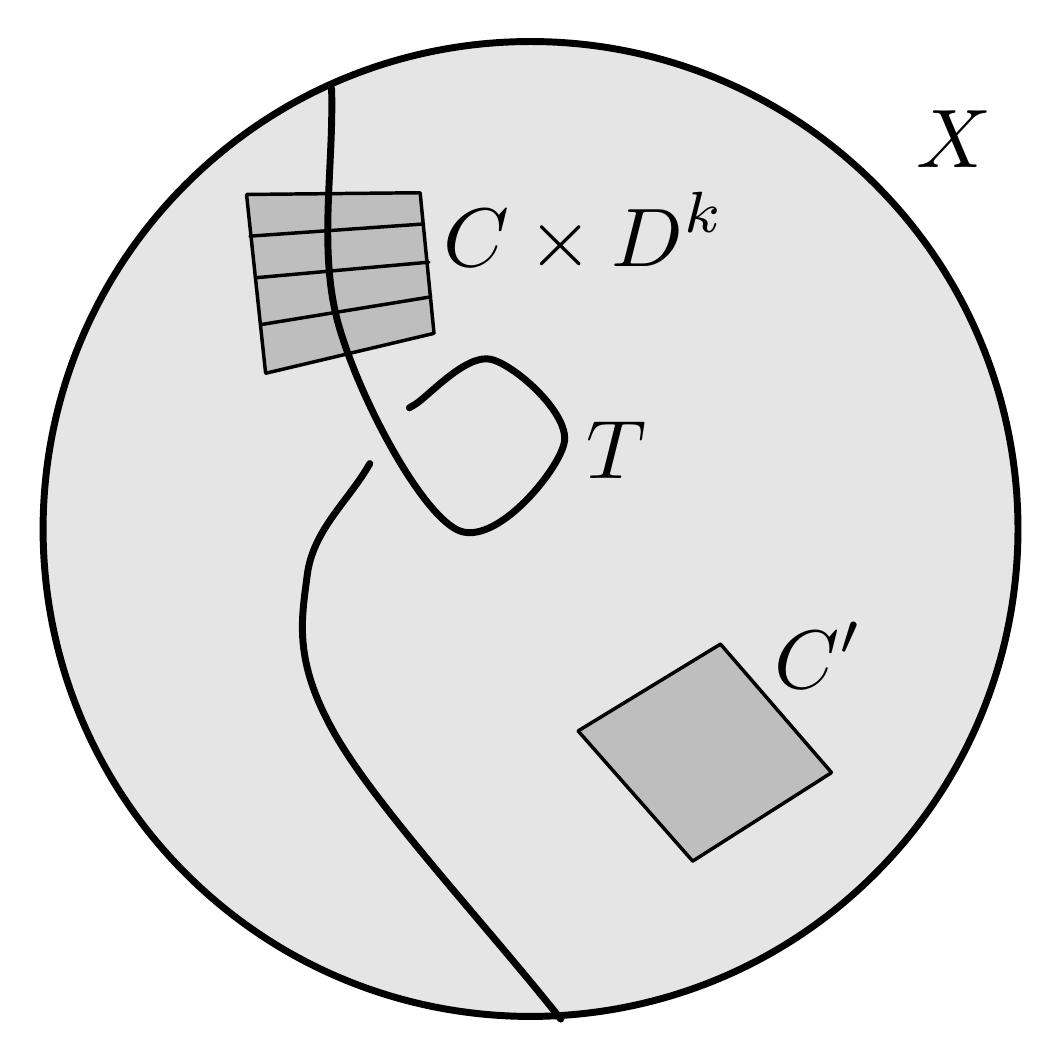}
}
\vspace{.2in}
\caption{The construction of $F(T)$ from a $G$-tangle $T$ in $\gtang{n}{k}(X)$. To a product cell in the disk-bundle neighbourhood of $T$ assign $\pi^*a \in \wcat{A}(C\times D^k)$ where $\pi : C \times D^k \to \sphere{k}$ is the composite of projection and collapse of the disk's boundary, and to other cells  assign the unit in $\wcat{A}(C')$. Then compose to obtain an element $F(T)$ in $\wcat{A}(X)$.}
\label{tangle hyp}
\end{figure}

Of course there are many technical issues. One must choose the disk-bundle neighbourhoods carefully. The cleanest approach is  to define $F$ on an auxiliary category of `$G$-tangles with disk-bundle neighbourhoods' and then show that the forgetful functor from this to $\gtang{n}{k}$ is an equivalence. In addition one must show that $F(T)$ is independent of the choice of cellular decomposition. This would follow from the existence of common subdivisions, so one must include sufficient hypotheses to ensure this property, for instance by fixing a PL structure and working with PL stratifications. Finally one needs to prove uniqueness.

The connection with transversal homotopy theory goes through more easily. Let $\thom{G}$ be the  Thom space of the universal bundle on the classifying space of $G$-bundles.  Stratify the Thom space by the classifying space (embedded as the zero section) and its complement. (In practice more care is required. One must choose a finite-dimensional smooth manifold model for the classifying space, which suffices to classify $G$-bundles over manifolds of dimension $\leq n$. Then one works with the `fat' Thom space, as defined in \cite[\S2]{MR2720181}, constructed from this model. The point is that to define transversal homotopy theory one needs to work with  a Whitney stratified manifold.) Proceeding as in \S\ref{transversal homotopy of spheres}, but with $0 \in \sphere{k}$ replaced by $BG \subset \thom{G}$ throughout, one obtains an $(n+k)$-equivalence
$$
\gtang{n}{k} \simeq \thcat{k,n+k}{\thom{G}}
$$
generalising that for $n=1$ in \cite[\S 5]{MR2720181}.

\end{document}